\documentclass[11pt]{article}

\usepackage{amssymb, amsmath,amsthm, mathrsfs, bbm, MnSymbol}
\usepackage[dvips]{color}
\usepackage{eucal}

\usepackage{graphicx}
\usepackage{epstopdf}

\newtheorem{theorem}{Theorem}
\newtheorem{proposition}[theorem]{Proposition}

\newtheorem{corollary}[theorem]{Corollary}

\theoremstyle{remark}
\newtheorem*{remark}{Remark}
\newtheorem{example}{Example}

\newenvironment{keywords}{\par}{\par}

\begin{document}

\title{Reaction-diffusion on metric graphs and conversion probability}

\author{ Renato Feres\footnote{Dept. of Mathematics, Washington University,
Campus Box 1146, St. Louis, MO 63130}, Matt Wallace\footnotemark[1]}


\maketitle

\begin{abstract}
We consider the following type of reaction-diffusion systems, motivated by a specific problem in the area of heterogeneous catalysis: 

A pulse of reactant gas of species $A$ is injected into a domain $U\subset \mathbb{R}^d$ that we refer to as the {\em reactor}. This domain is permeable to gas diffusion and chemically inert, except for a few relatively small regions, referred to as {\em active sites}. At the active sites, an irreversible first-order reaction $A\rightarrow B$ can occur. Part of the boundary of $U$ is designated as the {\em reactor exit}, out of which a mixture of reactant and product gases can be collected and analyzed for their composition. The rest of the boundary is chemically inert and impermeable to diffusion; we assume that instantaneous normal reflection occurs there. 

The central problem is then: Given the geometric parameters defining the configuration of the system, such as the shape of $U$, position of gas injection, location and shape of the active sites, and location of the reactor exit, find the (molar) fraction of product gas in the mixture after the reactor $U$ is fully evacuated. Under certain assumptions, this fraction can be identified with the {\em reaction probability}\----that is, with the probability that a single diffusing molecule of $A$ reacts before leaving through the exit. 

More specifically, we are interested in how this reaction probability depends on the rate constant $k$ of the reaction $A\rightarrow B$. After giving a stochastic formulation of the problem, we consider domains having the form of a network of thin tubes in which the active sites are located at some of the junctures. The main result of the paper is that in the limit, as the thin tubes approach curves in $\mathbb{R}^d$, reaction probability converges to functions of the point of gas injection that can be described fairly  explicitly in their dependence on the (appropriately rescaled) parameter $k$. Thus, we can use the simpler processes on metric graphs as model systems for more complicated behavior. We illustrate this with a number of analytic examples and one numerical example. \end{abstract}

\begin{keywords}
Reaction-diffusion, metric graphs.
\end{keywords}

\section{Introduction} 
Consider the system described schematically in Figure \ref{domain}. The bounded domain $U\subset \mathbb{R}^d$ represents a chemical reactor packed with inert granular particles permeable to gas diffusion. At the point $x\in U$, one injects a small pulse of gas molecules of species $A$. The pulse diffuses in $U$ with a certain diffusion coefficient $\mathcal{D}$ before escaping through the reactor exit, marked in the figure with a dashed line and denoted $\partial U_a$.  

Now suppose a certain small subset $\mathcal{A}$ of $U$ contains a catalytically active material promoting the irreversible first-order reaction $A\rightarrow B$ with {\em rate constant} $k$. We call $\mathcal{A}$ the {\em active zone} Then a certain fraction of the molecules in the initial pulse may interact with $\mathcal{A}$ and convert to $B$ before exiting the reactor. This number is called the {\em fractional conversion}. Motivated by certain problems in heterogeneous catalysis (see \cite{FYMB09, FCYG09} and references therein) we seek to determine the fractional conversion in terms of the rate constant $k$, the diffusion coefficient $\mathcal{D}$ and the geometric configuration of $\mathcal{A}$ and $U$.

\begin{figure}[htbp]
\begin{center}
 \includegraphics[width=3.0in]{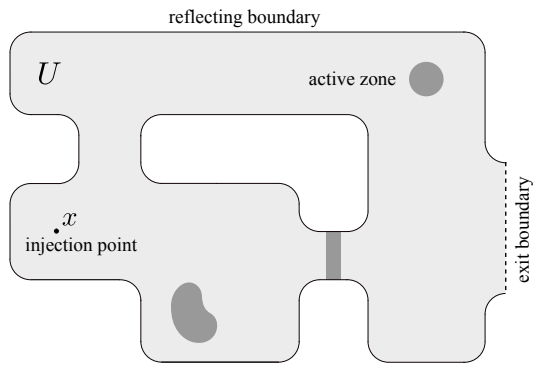}
\caption{\small  Schematic description of the chemical reactor. $U$ is the domain of diffusion; the shaded regions represent active sites, whose union is denoted $\mathcal{A}$. The boundary of $U$ is the union of a reflecting part $\partial U_r$ and an absorbing part $\partial U_a$.  The point $x$ is where reactant gas of species $A$ is injected. We are interested in the probability that a gas molecule starting at $x$ and undergoing Brownian motion will react before reaching $\partial U_a$.}
\label{domain}
\end{center}
\end{figure} 

The practical concern is that $k$ may be difficult to determine experimentally, because the reaction $A\rightarrow B$ may involve, at the microscopic level, a complex sequence of steps such as adsorption and surface reaction. On the other hand, measuring the fractional conversion is comparatively simple. Similarly, one can carefully control in a lab the size and shape of the reactor, the size and shape of the active zone, and the nature of the particulate matter supporting the diffusion (hence the diffusion coefficient $\mathcal{D}$). Determining $k$ in terms of these other variables is therefore of considerable interest.  

We intend to set up a probabilistic model for determining $k$. Let us briefly describe what physical assumptions we impose for the model to work. First, we assume that the pressure within $U$ is low enough, and the injected pulse is narrow enough, that the diffusion follows the {\it Knudsen regime.} In other words, the diffusion is driven at the microscopic level by tiny collisions between diffusing molecules and the particulate matter making up the reactor bed\----not by collisions between molecules themselves. Second, we assume the pulse is small enough that any reactions occurring within $U$ during the course of the experiment negligibly alter the composition of the active sites. These assumptions are physically reasonable and can in fact be arranged for in a lab; see \cite{GYZF10} or \cite{FSYW15} for more information.
 
With the set-up just described, write $\alpha(x)$ for the fractional conversion, or $\alpha_k(x)$ when necessary to emphasize dependence on the rate constant $k$. Because of our assumption about the Knudsenian character of the diffusion, we can identify $\alpha_k(x)$ with the {\em reaction probability;} that is, with the probability that a single molecule, entering the chamber at $x$ and undergoing reflecting Brownian motion in $U$ with diffusion coefficient $\mathcal{D}$, converts to $B$ before hitting $\partial U_a$. This identification is possible because every molecule's path in the reactor is independent of every other's. Therefore a single pulse of $N$ molecules is equivalent to $N$ independent pulses of one molecule. For this reason, we can model the reaction $A\rightarrow B$ by killing a reflecting Brownian motion (RBM) in $U$ with a rate function of the form $r=kq$ where $k$ is the reaction rate and $q$ is the indicator function on $\mathcal{A}$, or perhaps a smooth approximation thereof.  

By RBM in $U$, we mean a diffusion process $(X_t,\mathbb{P}_x)$ with sample paths in $\overline{U}$ and satisfying the stochastic differential equation 
\begin{equation} X_t - X_0  = \sigma B_t + \int_0^t \nu(X_s) \mathrm{d}\ell_s \label{eq:sk} \end{equation}
where $\sigma=\sqrt{2\mathcal{D}}$, $B_t$ is an ordinary $d$-dimensional Brownian motion, $\nu$ is the inward-pointing unit normal vector field on $\partial U$, and $\ell_s$ is a continuous increasing process (called the boundary local time) that increases only on $\{t : X_t\in\partial U\}$. The SDE \eqref{eq:sk} is called the {\em Skhorohod} or {\em semimartingale decomposition} of $(X_t)$. See, for example, \cite{BH91, C93} for more information about RBM's  that have a semimartingale decomposition. As for killing a diffusion, see \cite{RW1}. 

The existence of an RBM $(X_t)$ in $U$ with a decomposition \eqref{eq:sk} requires that $\partial U$ satisfy certain modest smoothness conditions. For example, $C^2$ smoothness is more than sufficient to ensure that \eqref{eq:sk} has a pathwise unique solution; see \cite{BB06, BB08, LS84, LS84, S87}. See \cite{IW} for general notions about stochastic differential equations. Henceforth we shall assume $\partial U$ is at least $C^2$ unless otherwise stated. Furthermore, we assume the first hitting time of $(X_t)$ to $\partial U_a$ is finite almost surely. As for $\mathcal{A}$, we assume for now that it's a compact subset of $U$ with a regular boundary, and impose further conditions when necessary. 

It will be convenient to introduce the {\em survival function} $\psi_k(x)$, defined as the probability that Brownian motion starting at $x$ reaches $\partial U_a$ without being killed (that is, without reacting). We also write $\psi(x)$ when $k$ is understood. Then the reaction probability we seek is $\alpha_k(x)=1-\psi_k(x)$. Our model for the reaction and diffusion within $U$ entails that \begin{equation}\label{eq:psi} \psi_k(x) = \mathbb{E}_x\left[ \exp\left\{ -\int_0^T kq(X_s)\mathrm{d}s \right\} \right] \end{equation} where $T$ is the first hitting time of $(X_t)$ to $\partial U_a$, that is $T:=\inf\{t>0\,:\,X_t\in\partial U_a\}$.

Functionals such as the right-hand side of \eqref{eq:psi} are well-known and have the following physical interpretation in terms of our catalysis problem: the probability that a gas molecule injected at $x$ and following the sample path $\omega$  {\em hasn't} reacted decreases exponentially  with the amount of time spent in the chemically active region  $\mathcal{A}$, the reaction constant $k$ being the exponential rate. Thus the expression $\exp\left\{-k\int_0^{T(\omega)} q(X_s(\omega))\, ds\right\}$ represents the probability, for that sample path, that the molecule does not react by the time it exits the reactor, and then $\psi_k(x)$ is the expected value of this probability. 

It is also well-known that $\psi_k(x)$ can be expressed as the solution to an elliptic boundary value problem in $U$. Our numerical illustration below is based on this fact:

\begin{proposition}
Suppose that $u$ is a function which is bounded and continuous in $\overline{U}$, satisfies \[ \mathcal{D}\Delta u -kqu=0\] in $U\setminus \partial \mathcal{A}$, together with the boundary conditions $\frac{\partial u}{\partial n}=0$ on $\partial U_r$ and  $u=1$ on $\partial U_a$. Suppose furthermore that $\partial\mathcal{A}$ is a set of zero potential for RBM in $\overline{U}$. Then $u=\psi$ within $U$. 
\end{proposition}

\begin{remark} Let $S=\inf\{t>0\,:\,X_t\in\mathcal{A}\}$ be the first hitting time of $X$ to $\mathcal{A}$. Then $\int_0^T kq(X)\mathrm{d}s$ is only positive on $(S<T)$ = the event that $X_t$ hits $\mathcal{A}$ at all before leaving through $\partial U_a$. Accordingly, letting $k\rightarrow\infty$ reduces \eqref{eq:psi} to $\psi_\infty(x)=\mathbb{P}_x\left[S<T\right]$ = the probability that, starting from $x\in U$, the RBM $X_t$ escapes through $\partial U_a$ without ever hitting $\mathcal{A}$. 
We write $P_{\mathrm{h}}(x)=1-\psi_\infty(x)$ for the complementary probability that $X_t$ hits $\mathcal{A}$ at all before leaving. This last probability can also be determined by solving a boundary value problem. Namely, if $v(x)$ solves $\Delta v=0$ in $V:=U\setminus\mathcal{A}$, together with the boundary conditions $u=1$ on $\partial\mathcal{A}$, $u=0$ on $\partial U_a$, and $\mathbf{n}\cdot\nabla v=0$ on $\partial V\setminus [\partial U_a\cup\partial\mathcal{A}]$ then $v(x)=P_{\mathrm{h}}(x)$. \end{remark}

Elsewhere we will address the existence of solutions satisfying the conditions in Proposition 1. For now, we simply assume that a solution exists and can be approximated using the Finite Element (FE) method. This is not a serious defect since we only intend to use Proposition 1 to corroborate numerically the predicted form of $\alpha_k(x)$ in Theorem 2 below. 

To understand what to expect, then, consider the case where $U$ is a line or a metric graph, and $\mathcal{A}=\{a\}$, a single point. Then simple manipulations involving the strong Markov property (described in detail below) show that $\alpha_k(x)$ factors as $\alpha_k(x)=P_{\mathrm{h}}(x)f(x,k)$ where $P_{\mathrm{h}}(x)$ is the hitting probability for $\mathcal{A}$ (described in the Remark under Proposition 1) and $f(x,k)$ is the reaction probability conditional on starting from $a$. Clearly $P_{\mathrm{h}}(x)$ doesn't depend on $k$, so that $f(k,x)$ is the quantity of interest. In this case, $f(x,k)$ is simply equal to $\alpha_k(a)$, and it follows from basic properties of local times (described in detail below) that \begin{equation}\label{typical}f(x,k)=\frac{\lambda k}{1+\lambda k}\end{equation} where $\lambda$ is a constant that doesn't depend on $k$. For this reason we regard the formula \eqref{typical} as ``typical'' behavior for situations where $\mathcal{A}$ can reasonably be thought of as a single point in the state space. This involves some coarse-graining, because single points are polar for Brownian motion. 

\vspace{0.1in}
\begin{figure}[htbp]
\begin{center}
 \includegraphics[width=2.2in]{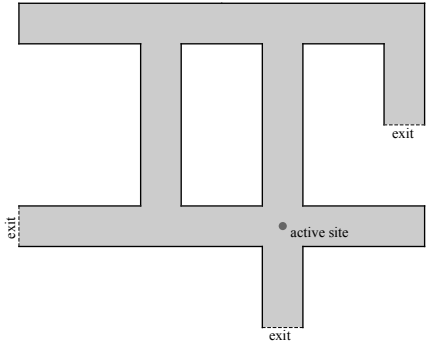}
\caption{\small  Example in dimension $2$ of a reactor domain with a small active site. }
\label{maize}
\end{center}
\end{figure} 

To illustrate this behavior, consider the domain in Figure \ref{maize}. Taking the length of the shortest side of the polygonal region as unit, the radius of the disc-shaped  active site is $0.1$. Here, $\mathcal{A}$ is not a single point, but is nevertheless small enough relative to $U$ that we can reasonably hope that a factorization of the form $\alpha_k(x)=P_{\mathrm{h}}(x)f(x,k)$ holds. That is, if we {\em define} $f(x,k)$ by \begin{equation} \label{typical2} f(x,k) :=  \frac{\alpha_k(x)}{P_{\mathrm{h}}(x)} \end{equation} and then compute the right-hand members of \eqref{typical2} using the boundary value problems determining $\alpha_k(x)$ and $P_{\mathrm{h}}(x)$, then we expect to see the behavior described in \eqref{typical}. 

To this end, we computed $f(x,k)$ as defined in \eqref{typical2} for a fixed $x$ and $50$ different equally-spaced values of $k\in[0,100]$ using FEniCS, and then ``solved for $\lambda$'' under the working assumption that \eqref{typical} holds. That is, we defined \[ \lambda := \frac{1}{k}\frac{f(x,k)}{1-f(x,k)} =\frac{1}{k}\frac{\psi_k(x)/P_{\mathrm{h}}(x)}{1-\psi_k(x)/P_{\mathrm{h}}(x)}\] and checked $\lambda$'s dependence on $k$. As expected, the resulting $\lambda$'s exhibited little dependence on $k$ . Denoting by $\lambda_{\text{\tiny max}}$ and $\lambda_{\text{\tiny min}}$ the maximum and minimum values of $\lambda$ over the range of $k$'s we looked at, we found $(\lambda_{\text{\tiny max}}-\lambda_{\text{\tiny min}})/\lambda_{\text{\tiny min}}< 0.005$ with mean value $\lambda=0.02342$. 

It will be shown that the above expression \eqref{typical} is indeed a typical approximation  of $f(x,k)$ when the active region is a small single site. Roughly speaking, $\lambda$ is the amount of time accumulated at $a$, starting from $a$, before departing through the reactor exit. In particular, $\lambda$ depends only on the geometry of the reactor and the diffusion coefficient, and not $k$. For more general active site configurations, more complicated but related approximation formulas apply.
  
We call the regions to which our main result will apply {\em fat graphs}. These are graph-like domains comprising a number of thin tubes and junctures that converge, in an appropriate sense, to an underlying metric graph $\Gamma=(\mathcal{V},\mathcal{E})$. Here, $\mathcal{E}$ denotes the edges of $\Gamma$, and $\mathcal{V}$ the vertices. Then each tube in the domain corresponds to an edge $e\in\mathcal{E}$, and each juncture corresponds to a vertex $v\in\mathcal{V}$. We assume that the active region $\mathcal{A}$ is concentrated in the juncture regions. The precise definitions are given in section \ref{fatgraph}. For now, figure \ref{fatgraphparameters} conveys a rough idea. 

\begin{figure}[htbp]
\begin{center}
 \includegraphics[width=3.5in]{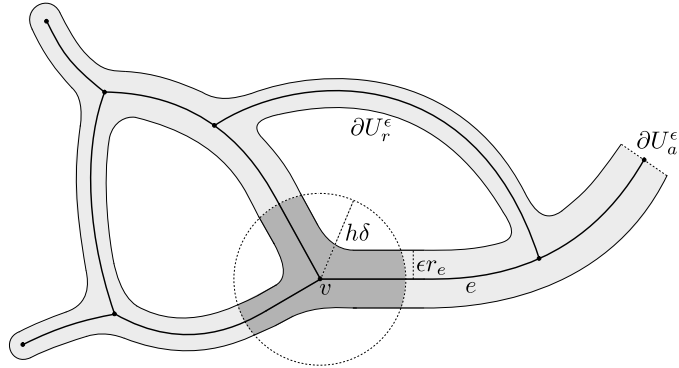}
\caption{\small A fat graph in $\mathbb{R}^d$ with its skeleton graph $\Gamma$.  Edges have relative radii $r_e$ and $\epsilon$ is a scaling parameter. Active sites, defined as regions where reaction can occur, have a relative radius $\delta$ with scaling parameter $h$. The reaction constant must also be scaled as $k^h=k/h$ as $h\downarrow 0$.}
\label{fatgraphparameters}
\end{center}
\end{figure} 
 
There are two scale parameters involved in the approximation: $\epsilon>0$, which gives the thickness of the tubes in the domain, and $h>0$, which gives the radius of the active zone around each active site. As $\epsilon\downarrow 0$, the domain collapses to the skeleton graph, and as $h\downarrow 0$, active sites collapse to vertices. At the same time, reaction activity, given by the rate constant $k^h=k/h$ must increase accordingly. 
  
Thus let $f^{\epsilon, h}(x,k/h)$ denote the quantity introduced above that gives the dependence of the reaction probability on the rate constant. That is, 
\[ f^{\epsilon, h}(x, k/h) := \frac{\psi_k(x)}{P_{\mathrm{h}}(x)}.\]
The main result is then the following. 

\begin{theorem} \label{theorem1}
Let $U^\epsilon$ be a bounded fat graph whose  skeleton metric graph $\Gamma$ has  finitely many vertices and edges. Let $\mathcal{C}$ be the set  of vertices corresponding to active sites of $U^\epsilon$ and $\delta$ the relative radius of the active sites. Define $\kappa:=k\delta/\cal{D}$,
  where $\mathcal{D}$ is the diffusion constant.  Let $|\mathcal{C}|$ denote the number of active vertices. Then, under the hypotheses of Proposition \ref{firstlimit} below, $$\lim_{h\rightarrow 0} \lim_{\epsilon\rightarrow 0} f^{\epsilon,h}(x,k/h)= 1-\sum_{v\in \mathcal{C}}p_v(x)\frac{\lambda_v(\kappa)}{\lambda(\kappa)},$$
  where  $p_v(x)$ is the probability that a diffusing particle starting from $x$ hits $\mathcal{C}$ for the first time at $v$, conditional on hitting $\mathcal{C}$ at all, and $\lambda(\kappa)$, $\lambda_v(\kappa)$ are polynomials in $\kappa$ of degree at most $|\mathcal{C}|$ and $|\mathcal{C}|-1$, respectively. The coefficients of $\lambda(\kappa),\lambda_v(\kappa)$ depend only on geometric properties of $\Gamma$: lengths of edges, degrees of vertices, location of exit   and active vertices. When $\mathcal{C}=\{v\}$, this limit reduces to 
  $$\lim_{h\rightarrow 0} \lim_{\epsilon\rightarrow 0} f^{\epsilon,h}(x,k/h)= \frac{\lambda \kappa}{1+\lambda \kappa}.$$
  \end{theorem}
  
  It would be interesting to find methods that could tell {\em a priori} how    the coefficients of $\lambda(\kappa)$ and $\lambda_v(\kappa)$
  can be expressed in terms of
  the geometry and topology of the graph, that is, as function of lengths, degrees, etc. A starting point would be to explore in a systematic way
 the properties of these polynomials   for families of graphs.
  Here we are content with showing a few examples only.  
 
    \vspace{0.1in}
\begin{figure}[htbp]
\begin{center}
 \includegraphics[width=4in]{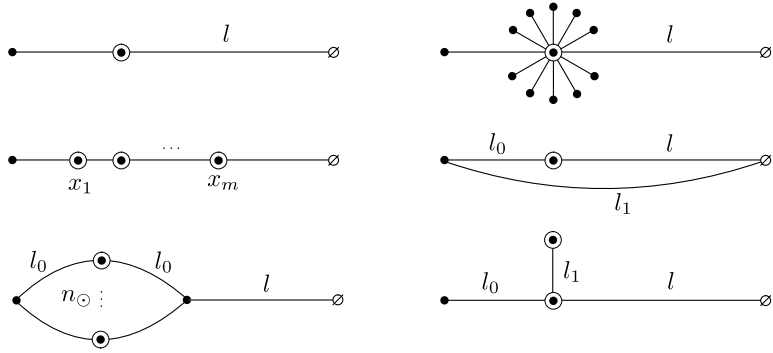}
\caption{\small A few examples of graphs for which the reaction probability is easily computed. The vertex types are: $\cdot=$ inert, $\odot=$ active, $\emptyset=$
exit. The labels $l_i$ indicate edge lengths and the $x_i$ on the middle graph of the left column are the coordinates of the active nodes.  On the third graph of the left column, $n_\odot$ is the number of vertices of type $\odot$.}
\label{diagrams}
\end{center}
\end{figure} 

The reaction probability for the examples of figure \ref{diagrams} are easily obtained by solving the linear system of equations indicated in section \ref{examples}.  In all cases we assume that the point of gas injection is the left-most vertex. The results are as follows.
\begin{enumerate}
\item Top graph on the left column: $$ \alpha(\kappa)= \frac{2l\kappa}{1+2l\kappa}.$$ Naturally, only the length of the edge between the active and exit vertices matters. On the other hand, if the initial edge is completely eliminated so that the starting point for Brownian motion is the active vertex, then the term $2l\kappa$ should be replaced with $l\kappa$.

\item Top graph on the right column: $$\alpha(\kappa)=\frac{\text{deg}(v)l\kappa}{1+\text{deg}(v)l\kappa} $$ where $\text{deg}(v)$ is the degree of the active vertex. As in the first example, only the length of the edge connecting the active vertex to the exit matters, but the number of shorter edges leading to inert vertices influences $\alpha$, regardless of those edges' lengths, so long as their are greater than zero.

\item Middle graph of the left column. Let $l_j=x_j-x_{j-1}$. The value of $\alpha(\kappa)$ can be obtained recursively as follows. Set $g_1=1$ and $g_{j+1}= g_j+ l_{j+1}\left(g_1+ \cdots + g_j\right) \kappa$ for $j=2, \dots, m$. (The vertex $\emptyset$ is at position $x_{m+1}$.) Then $$\alpha(\kappa)=\frac{g_{m+1}-1}{g_{m+1}}.$$ It is interesting to note the following property concerning optimal arrangement of active nodes: When $\kappa$ is small, $$\alpha(\kappa)=\left[(L-x_1)+\cdots+(L-x_m)\right]\kappa+ \text{higher order terms in } \kappa $$ where $L$ is the total  length of the graph. So if $\kappa$ is small, one maximizes reaction probability by clustering all the active nodes near the entrance point. On the other hand, for large values of $\kappa$, $$\alpha(\kappa) =l_2 l_3 \cdots l_{m+1} \kappa^m + \text{lower order tems in } \kappa.$$ One easily obtains that the coefficient of $\kappa^m$ attains a maximum when the active vertices are equally spaced.

\item Middle graph of the right column:
$$\alpha(\kappa)=\alpha(\infty)\frac{\lambda\kappa}{1+\lambda\kappa} $$
where $\alpha(\infty)= l_1/(l_0+l_1)$ and $\lambda=2{l(l_0+l_1)}/({l_0+l_1+l}).$
\item Bottom graph of left column:
$$\alpha(\kappa)=\frac{2(l_0+n_\odot l)\kappa}{1+2(l_0+n_\odot l)\kappa}.$$
Despite having multiple active vertices, this system behaves, in its dependence on $\kappa$, like one with a single active vertex.

\item Bottom graph of right column:
$$ \alpha(\kappa)= \frac{3l\left(1+l_1\kappa\right)\kappa}{1+3l\left(1+l_1\kappa\right)\kappa}.$$
As expected for a system with two active vertices, this is the quotient of second degree polynomials in $\kappa$.
\end{enumerate}
 
\section{Metric graphs and fat-graph domains}\label{fatgraph}

We introduce here notation and terminology concerning {\em metric graphs,} together with an associated class of domains in $\mathbb{R}^d$ we call {\em fat graphs.} Basically, a metric graph is an abstract graph realized as a collection of curvilinear segments; a fat graph is a neighborhood of $\Gamma$ whose boundary satisfies some smoothness conditions. 

More precisely, let $\mathcal{G}=(\mathcal{V}, \mathcal{E})$ be an abstract graph with vertex set $\mathcal{V}$ and edge set $\mathcal{E}$. We assume the edges in $\mathcal{E}$ are oriented, and denote by $s, t: \mathcal{E}\rightarrow \mathcal{V}$ the source and target vertex functions. For each edge $e\in \mathcal{E}$ let $\bar{e}$ be the inverse of $e$, which is the same edge given the opposite orientation. Thus $s(\bar{e})=t(e)$ and $t(\bar{e})=s(e)$. We assume that $\mathcal{E}$ is closed under the inverse operation. We also assume that $|\mathcal{E}|,|\mathcal{V}|<\infty$  where $|\cdot|$ indicates the cardinality. 

Associate to each  $v\in \mathcal{V}$ a point in $\mathbb{R}^d$, still denoted $v$ (so that we now think of $\mathcal{V}$ as a subset of the Euclidean space) and to each $e\in \mathcal{E}$ a smooth curve $\gamma_e:[0,l_e]\rightarrow \mathbb{R}^d$, parametrized by arclength, such that $\gamma_e(0)=s(e)$ and
  $\gamma_e(l_e)=t(e)$. Thus, $l_e$ is the length of the curvilinear segment $E_e:=\gamma_e([0,l_e])$. We assume that $\max_e l_e<\infty$ and write $l_e=|e|=|E_e|$. We also set $\gamma_{\bar{e}}(s):=\gamma_e(|e|-s)$. 
    
The union $\Gamma:=\bigcup_{e\in\mathcal{E}} E_e$ will be called a {\em metric graph}, denoted $\Gamma$. With slight notational abuse we also write $\Gamma^\circ$ for $\bigcup_{e\in\mathcal{E}} E^\circ_v$, where $E_v^\circ=\gamma_e((0,l_e))$. Thus, $\Gamma^\circ=\Gamma\setminus\mathcal{V}$. The word {\em metric} refers to the natural distance defined by minimizing a path between two points. This induced metric makes $\Gamma$ into a separable metric space. 

Each edge has a natural coordinate $y_e=\gamma_e^{-1}:E_e\rightarrow I_e$. By means of this coordinate we can identify functions on $E_e$ with functions on $[0,l_e]$. Similarly, we can identify a function $f:\Gamma\rightarrow\mathbb{R}$ with a collection of functions on the various coordinate intervals $[0,l_e]$ by defining $f_e(s)=f(\gamma_e(s))$ for $s\in[0,l_e]$. Thus we have an obvious way of checking that $f$ is continuous: each $f_e$ must be continuous on $(0,l_e)$ in the ordinary sense, and the extensions to $[0,l_e]$ must agree at the vertices, in the sense that $f_{e_1}(l_{e_1})=f_{e_2}(0)$ whenever $t(e_1)=s(e_2)$. The set of continuous functions on $\Gamma$ is denoted $C(\Gamma)$.

Similarly, we can define the derivative of $f\in C(\Gamma)$ at a point $x=\gamma_e(s)\in E^\circ_v$ by $(f\circ\gamma_e)'(s)$, when this derivative exists in the ordinary sense. At a vertex $v$, we define the one-sided derivatives \[ (D_ef)(v) := \lim_{s\downarrow 0} \frac{f(\gamma_e(s))-f(v)}{s}\] when the limit exists. Thus $D_ef(v)$ is the directional derivative of $f$ at $v$, pointing into $e$. Clearly this definition only makes sense for $e\in\mathcal{E}$ such that $s(e)=v$. Note that, in general, we do not require that the directional derivatives at $v$ all agree. 

We now define the fat graph $U^\epsilon$ as a union of certain $\epsilon$-tubes $U_e^\epsilon$ (one for each $e\in\mathcal{E}$) together with $\epsilon$-junctures $J_v^\epsilon$ (one for each $v\in\mathcal{V}$). Some care is required in formulating the definitions; however, Figure \ref{fatgraphparameters} should convey the right idea. The following conditions ensure that the construction works properly:
\begin{enumerate}
\item If $e, e'$ are any two edges such that $s(e)=s(e')=v$, then $\mathbf{T}_e(v)\neq \mathbf{T}_{e'}(v)$.
\item \label{straight} There exists a $r_0>0$ such that    $\mathbf{T}_e'(s)=0$ for  $s\in [0,r_0]$ and all $e\in \mathcal{E}$.
\item \label{nonintersect} The curves comprised by $\Gamma$ do not intersect each other and have no self-intersection.
\item \label{smooth} Each $\gamma_e$ is $C^3$.
\end{enumerate}

Now fix an $\epsilon>0$. We begin by defining the $\epsilon$-tubes  $U^\epsilon_e$ corresponding to the edges $E_e$ in $\Gamma$. For each $e\in\mathcal{E}$, let there be given a {\em relative radius} $r_e>0$; and then, for each $\epsilon>0$ sufficiently small, let $U^\epsilon_e$ be a tubular neighborhood of $E_e^\circ$ with cross-sectional radius $r_e\epsilon$. According to \cite{F84}, $U^\epsilon_e$ has the {\em nearest point property} with respect to $E_e^\circ$, meaning each $x\in U^\epsilon$ has a unique point $\pi^\epsilon(x)\in E_e^\circ$ nearest to $x$, and the induced mapping $\pi^\epsilon:U^\epsilon_e\rightarrow E_e^\circ$ is a $C^2$ submersion. Furthermore, the distance function \[ x\mapsto d(x,E_e^\circ) := \inf\{y\in E_e^\circ\,:\, \|x - y\|\} \] is $C^3$ near $E_e^\circ$ (because $\gamma_e$ is $C^3$\----see \cite{F84}), so that the unit normal vector field, given by \[ \mathbf{n}^\epsilon(x)=\nabla d(x,E_e)/\|\nabla d(x,E_e)\|,\quad x\in \partial U^\epsilon_e,\] is $C^2$ away from the ends of $U^\epsilon_e$. 

Now, the full $U_e^\epsilon$'s may intersect near the vertices. For this reason, we use assumptions \ref{straight} and \ref{nonintersect} from above to introduce constants $c_0>0$ and $\epsilon_0>0$ which depend only on $\Gamma$ and have the following property: if each $U_e^\epsilon$ is shrunken by a length $c_0\epsilon$ at its ends, then $U_{e_1}^\epsilon\cap U_{e_2}^\epsilon=\varnothing$ whenever $e_1,e_2\in \mathcal{E}$ are distinct and $\epsilon\leqslant \epsilon_0$. In summary, we have a family $\{ U_e^\epsilon \, :\, e\in\mathcal E,\,\epsilon<\epsilon_0\}$ such that: (i) for fixed $\epsilon\leqslant\epsilon_0$ the $U_e^\epsilon$'s are pairwise disjoint; (ii) $\bigcup_{e\in\mathcal{E}} U_e^\epsilon$ has the nearest point property with respect to $\Gamma^\circ$, and the natural projection $\pi^\epsilon$ onto $\Gamma^\circ$ is $C^2$; (iii) $\bigcup_{e\in\mathcal{E}}U_e^\epsilon\downarrow \Gamma^\circ$ as $\epsilon\downarrow 0$. 

Next, we define the juncture regions $J_v^\epsilon$. For each $v\in \mathcal{V}$ let there be given a ``template'' region such that \[ J_v\cup\bigcup_{e:s(e)=v} U_e^{\epsilon_0}\] is simply-connected and has a boundary smooth enough that the unit normal vector field is still $C^2$ away from the ends. In other words, $J_v$ lines up smoothly with its incident $\epsilon_0$-tubes.  It is clear that such a $J_v$ exists. It is also clear that we can shrink or enlarge $J_v$ so that $J_v\cap U_e^{\epsilon_0}$ is a cylindrical wafer of height $(r_0-c_0)\epsilon_0>0$ and radius $r_e\epsilon_0$, for each $e\in\mathcal{E}$. Then $J_v^\epsilon$ is defined in an obvious way by scaling down homothetically: \[ J^\epsilon_v:=\left\{x\in \mathbb{R}^d:   v+\epsilon^{-1}(x-v)\in   J_v\right\}.\] 

\begin{figure}[htbp]
\begin{center}
 \includegraphics[width=3.5in]{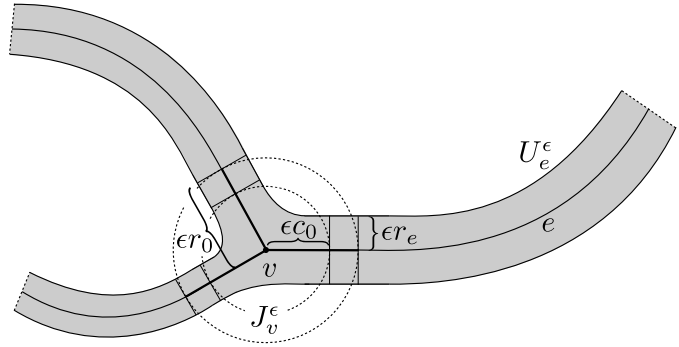}
\caption{\small  The $\epsilon$-juncture $J^\epsilon_v$ at $v$ in part of a  fat graph domain $U^\epsilon$ with skeleton $\Gamma$. The $r_e$ are
the relative radii of the tube cross sections. Note that the part of the edge curve $\gamma_e$  contained  in $J^\epsilon_v$ is straight.}
\label{juncture}
\end{center}
\end{figure}

Finally, with the $\epsilon$-tubes and $\epsilon$-junctures as above, we define the  {\em fat graph} with fatness parameter $\epsilon$ and relative radii $r_e$ as \[ U^\epsilon = \left[\bigcup_{v\in \mathcal V} J_v^\epsilon \right]\cup \left[ \bigcup_{e\in\mathcal{E}} U_e^\epsilon \right]. \]
From the construction, the unit normal vector field on $\partial U^\epsilon$ is $C^2$, and $U^\epsilon\downarrow \Gamma$ as $\epsilon\downarrow 0$. It is also clear that we can extend the projections $\pi^\epsilon$ to be defined inside the $J_v^\epsilon$'s in such a way that $\pi^\epsilon:U^\epsilon\rightarrow\Gamma$ is continuous, and $$ \sup_{v\in \mathcal{V}} \sup_{x\in J^\epsilon_v} \|\pi^\epsilon(x)-v\|=O(\epsilon).$$ Furthermore, we have that \begin{equation} \label{proj} \sup_{x\in  U^\epsilon} \|\pi^\epsilon(x) - x\| = O(\epsilon)\end{equation} uniformly in $x\in U^\epsilon$.

\section{Limit of diffusions from the fat graph to its graph skeleton}\label{fattothin}

In this section we review some background material related to diffusions on fat-graph domains and their limits as the fat graphs shrink down to their underlying metric graphs. The main result quoted here comes from \cite{AK12}, which extends the earlier work \cite{FW93}, with modifications required for our needs. 

Let $\Gamma$ be a metric graph in $\mathbb{R}^d$, and $U^\epsilon$ a family of fat graphs with skeleton $\Gamma$, fatness parameter $\epsilon>0$, and relative radii $r_e$ ($e\in\mathcal{E}$) as defined in Section \ref{fatgraph}. Let $a=\left(a^{ij}(x)\right)$ be a $d\times d$ matrix-valued function on $\mathbb{R}^d$ and $b=\left(b^i(x)\right)$ a vector field in $\mathbb{R}^d$. We assume that $a$ is uniformly positive definite, and that $a$ and $b$ are both bounded and Lipschitz continuous in $\mathbb{R}^d$. Define the differential operator $$L=\frac12 \sum_{i,j} a^{ij}(x) D_iD_j +  \sum_i b^i(x) D_i,$$ where $D_i$ is partial differentiation in $x_i$.  

Writing $a(x)=\sigma(x)\sigma^t(x)$, with $\sigma$ assumed positive-definite and Lipschitz continuous, consider the stochastic differential equation 
\begin{equation}\label{reflecting} X_{t}^\epsilon-X_0^\epsilon= \int_0^{t} \sigma\left(X_s^\epsilon\right)\, dW_s + \int_0^{t} b\left(X^\epsilon_s\right)\, ds
+\int_0^{t}\mathbf{n}^\epsilon\left(X_s^\epsilon\right)\, d\ell_s^\epsilon,\end{equation}
in which $W$ is a $d$-dimensional Brownian motion, $\mathbf{n}$  is the inward pointing unit normal vector field on $\partial U^\epsilon$ and $(\ell_t^\epsilon)$ is a continuous increasing process adapted to the filtration of $X$ and increasing only on the set $\{ t\,:\, X_t\in\partial U^\epsilon\}$. Since $\mathbf{n}^\epsilon$ is $C^2$, the geometric conditions in \cite{S87} are clearly met. Therefore \eqref{reflecting} admits a strong solution, in the sense of \cite{IW}. 
  
Let us describe the behavior of $(X^\epsilon_t)$ as $\epsilon\downarrow 0$. First, define an operator on $C(\Gamma)$ as follows. For each $e\in \mathcal{E}$, let $C_b^2(E_e^\circ)$ denote the space of functions on $E_e^\circ$ with two bounded and continuous derivatives; and then let $L_e$ act on $f\in C_b^2(E_e^\circ)$ as 
$$(L_ef)(x)= \frac12\left\| \sigma^t(x)\mathbf{T}_e(x)\right\|^2  f''_e(s) + \left[\left\langle b(x), \mathbf{T}_e(x) \right\rangle  + 
\left\langle a(x)\mathbf{T}'_e(x), \mathbf{T}_e(x)\right\rangle\right]f'_e(s), $$
where  $x=\gamma_e(s)$, $\mathbf{T}'_e(x)=\left(\mathbf{T}_e\circ\gamma_e\right)'(s)$ and $\langle\cdot,\cdot\rangle$ is the ordinary Euclidean inner product. To ``paste together'' the $L_e$'s, we must specify what happens at the vertices. Thus, we define $\mathcal{D}(L_\Gamma)$ to contain those functions $f\in C(\Gamma)$ for which:
\begin{enumerate}
\item $f_e\in C^2_b(E_e^\circ)$ for each $e\in\mathcal{E}$;
\item for each $v\in \mathcal{V}$,  and $e\in \mathcal{E}$ such that $v=s(e)$, the one-sided limits $\lim_{s\downarrow 0}(L_ef)(\gamma_e(s))$ exist and have
a common value, denoted $(Lf)(v)$;
\item  for each $v\in \mathcal{V}$,
\begin{equation}\label{pev} \sum_{e:s(e)=v} p_v(e) (D_ef)(v)=0\end{equation}
where the numbers $p_v(e)$ are defined by 
 \begin{equation}\label{defpev} p_v(e):=\frac{r_e^{d-1}}{\sum_{e':s(e')=v} r_{e'}^{d-1}}.\end{equation} 
\end{enumerate}
Then, for $f\in \mathcal{D}\left(L_\Gamma\right)$ as above, we define $L_{\Gamma}f$ by
$$(L_\Gamma)(x)=\begin{cases}
(L_ef_e)(y_e(x)) & \text{ if } x\in E_e^\circ\\
\lim_{y\rightarrow x}(L_ef)(y) & \text{ if } x\in \mathcal{V}.
\end{cases}
$$  

According to Theorem 3.1 in \cite{FW93}, there is a diffusion process $(X_t)$ on $\Gamma$ generated by $(L_\Gamma,\mathcal{D}(L_\Gamma))$. On the other hand, Theorem 4.2 in \cite{AK12} says that $(X^{\epsilon}_t)$ converges in distribution to $(X_t)$ as $\epsilon\downarrow 0$ if $X^\epsilon_0$ converges in distribution to a $\Gamma$-valued random variable. Since pathwise uniqueness holds for \eqref{reflecting} under our assumptions, ``in distribution'' can be replaced by ``almost sure'' in the preceding sentence. In particular, this is true if the starting point $X_0^\epsilon$ is a point $x\in\Gamma$ for all $\epsilon>0$. 

\begin{theorem}\label{diffusionlimit} Let $U^\epsilon$ be a fat graph with skeleton   $\Gamma$.
Assume that $X_0^\epsilon=x$ for all $\epsilon>0$ and some $x\in\Gamma$. Then $X^\epsilon$ converges in distribution to $X$ with initial value $X_0=x$. If pathwise uniqueness holds in \eqref{reflecting}, then $X^\epsilon$ converges a.s. to $X$.

\begin{remark} The coefficients $p_v(e)$ appearing above have the following probabilistic significance: if $\tau_\delta$ is the first exit time of $X_t$ from a ball of radius $\delta$ at $v$, then $p_v(e)=\lim_{\delta\downarrow 0}\mathbb{P}_v[X(\tau_\delta)\in E_e]$.  \end{remark}
\end{theorem}
\section{Diffusions with killing} Let $U^\epsilon$ be a fat-graph domain with skeleton $\Gamma$. The terminology and assumptions of Sections \ref{fatgraph} and \ref{fattothin} remain in force. 

We introduce a function $r(x)$ which represents the {\em rate of killing} of a diffusion in $U^\epsilon$ or in $\Gamma$.  This function is assumed to depend on a parameter $h>0$ whose role will become clear later; roughly,  $r^h$ will ``collapse'' to a vertex condition when $h$ goes to $0$. But for the moment we suppress $h$ in order to simplify the notation.

Thus, let $r:\mathbb{R}^d\rightarrow [0,\infty)$ be a non-negative measurable function, representing the chemical reaction rate. We are interested in processes obtained from $X^\epsilon$ on $U^\epsilon$ (resp. $X$ on $\Gamma$) by killing using the rate function $r$. Loosely, killing with rate $r$ means forming new processes $Y^\epsilon$ (resp. $Y$) which behave like $X^\epsilon$ (resp. $X$) until $\int_0^t r(X^\epsilon_s)\,ds$ (resp. $\int_0^t r(X_s)\,ds$) exceed an independent exponential random time; after this time, they are sent to a cemetery state $\triangle$. In this situation, the Feynman-Kac formula says that $Y^\epsilon$ and $Y$ have extended generators
$$L^rf= Lf-r(x)f \ \ \ \text{ and }\ \ \  L^r_\Gamma f=L_\Gamma f- r(x) f,$$
respectively, where $L$ and $L_\Gamma$ are as defined in section \ref{fattothin} and the domains are the same. Furthermore, the semigroup of $Y^\epsilon$ acts on functions as \[ P_t^\epsilon f(x) = \mathbb{E}_x\left[\exp\left\{ -\int_0^t r(X^\epsilon_s)\,ds\right\} f(X_t^\epsilon)\right] \] with a similar statement holding for $Y$. In particular, $P_t^\epsilon1(x)=\mathbb{E}_x\left[\exp\left\{-\int_0^t r(X_s^\epsilon)\,ds\right\}\right]$ is the probability that $Y^\epsilon$ is still in $U^\epsilon$, i.e. still ``alive'' at time $t$. With this in mind, we define the {\em survival function} of the process  at  $x\in U^\epsilon$ as
 \begin{equation}\label{fk1}\psi^{\epsilon}(x):=
 \mathbb{E}_x \left[\exp\left\{-\int_0^{T_a^\epsilon} r\left(X^\epsilon_s\right)\, ds\right\} \right].  \end{equation} Here, $T_a^\epsilon$ is a random time which can be thought of as the time that $X^\epsilon$ is absorbed at the reactor, as we now explain. 
 
 \begin{figure}[htbp]
\begin{center}
 \includegraphics[width=3in]{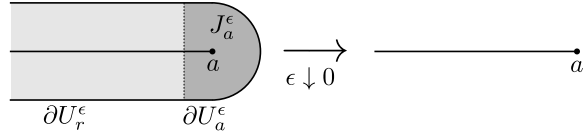}
\caption{\small  A portion of the absorbing part of the boundary of $U^\epsilon$ and its limit in $\Gamma$. The process $X^\epsilon$ in $U^\epsilon$ is killed when it reaches the dotted line separating
   the regions indicated in light and dark shading. The limiting process $X$ on $\Gamma$ is killed when it reaches $a$. }
\label{end}
\end{center}
\end{figure} 
 
Let there be given a certain subset $\mathcal{V}_a=\{a_j\}\subset\mathcal{V}$ of degree 1 vertices in $\Gamma$ to be regarded as the reactor exit. Write (with abuse of notation) $\partial U_a^\epsilon$ for the closure of the union of the corresponding juncture regions in $U^\epsilon$, and $\partial U_r^\epsilon$ for the rest of the boundary of $U^\epsilon$. We call $\partial U_a^\epsilon$ the {\em absorbing} part of the boundary and $\partial U_r^\epsilon$ the {\em reflecting} part. Then $\partial U_a^\epsilon\downarrow \mathcal{V}_a$ as $\epsilon\downarrow 0$. (See Figure \ref{end}.) Let $T_a^\epsilon$ be the first hitting time of $X^\epsilon$ to $\partial U_a^\epsilon$, and $T_a$ the first hitting time of $X$ to $\mathcal{V}_a$. For now, we {\em assume} that $T_a^\epsilon$ and $T_a$ are finite a.s., and that $T^\epsilon \rightarrow T$ a.s. Under these circumstances, we have the following:

 \begin{proposition} \label{firstlimit}
Suppose that $T_a^\epsilon$ and $T$ are a.s. finite and that $T_a^\epsilon\rightarrow T_a$ a.s. If $x\in \Gamma$ then
$$\psi^{\epsilon,h}(x)\rightarrow \mathbb{E}_x\left[ \exp \left\{ -\int_0^{T_a} r^h\left(X_s\right)\, ds\right\}\right] $$
as $\epsilon\downarrow 0$, where $T_a$ is the hitting time to $\mathcal{V}_a$ of the limiting  process $X$.
 \end{proposition}
    
With some additional smoothness on $r^h$, this conclusion can also be recast as a boundary value problem, which will be useful for computations. The following conclusion is standard. See, for example, \cite{B98}.
    
    \begin{corollary}
Suppose that $u$ is a bounded and continuous function on $\overline{U^\epsilon}$ that solves the equation $Lu-r^h u=0$ in $U^\epsilon$, together with the boundary conditions 
$$
\frac{\partial u}{\partial n}= 0  \text{ on } \partial  U_r^\epsilon \text{ and }
u=1  \text{ on } \partial U_a^\epsilon.
$$
Then $u=\psi^{\epsilon, h}$, with $\psi^{\epsilon,h}$ defined as in \eqref{fk1}, and $$\psi^{\epsilon, h}(x)\rightarrow \psi^h(x)$$
for every $x\in \Gamma$, where $\psi^h$ is the solution to the corresponding ordinary differential equation  on the graph $\Gamma$, namely: $L_\Gamma u-r^hu=0$
with boundary condition $u=1$ on $\mathcal{V}_a$. 
    \end{corollary}
  
\section{Collapsing  active zones towards vertices of $\Gamma$}

In the previous sections we described a conservative diffusion $X=(X_t)$ on a metric graph $\Gamma$ generated by an operator $L_\Gamma$ acting on a domain characterized by the vertex condition \eqref{pev}. This $X$ was obtained by collapsing the conservative diffusion $X^\epsilon$ on the fat graph $U^\epsilon$ down to $\Gamma$ as $\epsilon\downarrow 0$. We also described a nonconservative process $Y^h=(Y^h_t)$ obtained by killing $X$ using the rate function $r^h$. The killed process $Y^h$ has generator $L_\Gamma-r^h$ acting on the same domain. The function $r^h$ represents the rate of chemical activity on the active zones. 

Now, we wish to collapse the active zones, i.e. the regions in $\Gamma$ where $r^h$ is positive, to a collection of vertices, as $h\downarrow 0$. We have in mind that all the chemical activity is concentrated on the active vertices as $h\downarrow0$, while the killing rate is increased towards $\infty$ in such a way that the overall effect is the same in the limit. 

For concreteness, we assume that the operator $L_\Gamma$ acts as $\mathcal{D}\frac{d^2}{d y_e^2}$ on each edge $E_e$, where $\mathcal{D}$ is the diffusion coefficient; also, that the active regions are balls $B(v,h\delta)$ centered at  some of the vertices with radius $h\delta$, where $\delta>0$ is a fixed number with the units of distance, and $h>0$ is a dimensionless parameter. Thus, each $B(v,h\delta)$ is a star-shaped neighborhood of $v$ consisting of $v$ and a union of segments of edges of length $h\delta$ for each edge issuing from $v$. The killing rate function $r^h(x)$ will  then be  assumed to have the form
\begin{equation}\label{rateh}r^h(x)=\sum_{v\in \mathcal{V}}\frac{k_v}h \mathbbm{1}_{B(v,h\delta)}(x)\end{equation}
where $k_v$ is either $k$ or $0$ depending on whether that vertex is to be considered active or not. 
It is known (as explained in Remark 2.5 of \cite{FS00}) that,  as $h\downarrow 0$, 
$$\mathbb{E}_x\left[ \exp \left\{ -\int_0^{T_a} r^h\left(X_s\right)\, ds\right\}\right] \rightarrow \mathbb{E}_x\left[ \exp\left\{-\sum_{v\in \mathcal{V}} \kappa_v \ell_v(T_a)\right\}\right],$$
where $\kappa_v= k_v\delta/\mathcal{D}$ and $\ell_v(T_a)$ is the semimartingale local time (as defined in \cite{FS00}) at $v$ evaluated at the hitting time $T_a$. 

Write $\mathcal{C}$ for the set of active vertices in $\mathcal{V}$, and let $\ell_\mathcal{C}(t)$ denote the local time accumulated at $\mathcal{C}$ up to time $t$. In other words, $\ell_\mathcal{C}(t)$ is the sum of the $\ell_v(t)$ for which $k_v$ is not zero. Then the limit in the above expression becomes $\mathbb{E}_x\left[\exp\left\{ -\kappa \ell_\mathcal{C}(T_a) \right\}\right]$ where $\kappa=k\delta/\mathcal{D}$. 
Also, the killed processes $Y^h$ converge to a process $Y$ which is obtained from $X$ as follows: run $X$ until $\kappa\ell_\mathcal{C}(t)$ exceeds an independent rate $1$ exponential; after this time, send $X$ to the cemetery state $\triangle$. In other words, kill $X$ using the local time $\ell_\mathcal{C}(t)$ rather than the integral $\int_0^t r^h(X_s)\,ds$. Then the new process $Y$ is has a generator which is defined in the same way as $L_\Gamma$ for functions on $\Gamma^\circ$. However, the domain of $Y$'s generator is characterized by a different vertex condition, as explained in the following: 

\begin{proposition}

Let $r^{\epsilon,h}(x)$ be positive bounded functions on $U^\epsilon$ with the property that $$\lim_{\epsilon\downarrow 0} r^{\epsilon,h}(x)=r^h(x)$$ for all $x\in\Gamma$,  where $r^h(x)$ is defined at \eqref{rateh}. Let $\psi^{\epsilon,h}(x)$ and $\psi(x)$ be the survival functions associated with $Y^{\epsilon,h}$ and $Y$, respectively; that is,  \begin{align*}\psi^{\epsilon,h}(x)&=\mathbb{E}_x\left[ \exp\left\{ -\kappa\int_0^{T_a^\epsilon} r^{h,\epsilon}(X^\epsilon_s)\,ds\right\}\right],\\
\psi(x) &= \mathbb{E}_x\left[ \exp\left\{ -\kappa\ell_\mathcal{C}(T_a)\right\}\right]. \end{align*}If the conditions of Proposition \ref{firstlimit} are met, then
$$\psi(x)=\lim_{h\rightarrow 0}\lim_{\epsilon\rightarrow0} \psi^{\epsilon, h}(x)$$
where $\kappa=k\delta/\mathcal{D}$, as defined above. Furthermore, the generator of the limiting process has a domain characterized by the vertex condition
\begin{equation}\label{vcond}\sum_{e:s(e)=v} p_v(e) (D_eu)(v)=\kappa_v u(v) \end{equation} where $\kappa_v=\kappa$ if $v\in\mathcal{C}$ and $\kappa_v=0$ if $v\in\mathcal{V}\setminus\mathcal{C}$.\end{proposition}

\begin{proof}[Sketch of proof] The limit statement has already been discussed. As for the vertex condition, we will show that a function in the domain of $Y$'s generator must satisfy \eqref{vcond}.  Thus, let $(\hat{L},D(\hat{L}))$ be the generator of $Y$. Fix an active vertex $v\in\mathcal{C}$. Let $U_\delta=B(v,\delta)$ and $\hat{\tau}_\delta$ the first exit time of $Y_t$ from $U_\delta$. Also, write $\hat{\mathbb{E}}_v$ for the law of $Y_t$ started at $v$. If $F\in D(\hat{L})$ then Dynkin's formula (see \cite{RW1} Ch. 3) reads: \[ (\hat{L}F)(v)= \lim_{\delta\downarrow 0} \frac{\hat{\mathbb{E}}_v[F(Y(\hat{\tau}_\delta))] - F(v)}{\hat{\mathbb{E}}_v[\hat{\tau}_\delta]}\]
Let $\mathbf{e}$ be an independent rate 1 exponential (we can always enlarge the probability space to accommodate such a variable) and then $\hat{\zeta}=\inf\{ t>0 : \kappa\ell_\mathcal{C}(t) > \mathbf{e}\}$. Thus, $\hat{\zeta}$ is the time that $Y_t$ jumps from $\Gamma$ to the cemetery state $\triangle$. Evidently $\hat{\tau}_v=\tau_\delta\wedge\hat{\zeta}$ where $\tau_\delta$ is the first exit time of $X_t$ from $U_\delta$. In the denominator, then, we have \[ \hat{\mathbb{E}}_v[\hat{\tau}_v] \sim \mathbb{E}_v[\tau_\delta] = \frac{\delta^2}{2\mathcal{D}}\] as $\delta\downarrow 0$. On the other hand, $Y_t$ leaves $U_\delta$ either through a point $(\delta,e)$ (using local coordinates) if $\tau_\delta<\hat{\zeta}$ or by a jump to $\triangle$ if $\hat{\zeta}\leqslant\tau_\delta$. Since $F(\triangle)=0$ in the latter case, we have in the numerator, in view of the Remark under Theorem \ref{diffusionlimit},  \begin{align*} \hat{\mathbb{E}}_v[F(Y(\hat{\tau}_\delta))]-F(v) &= \sum_e F_e(\delta)\mathbb{P}_v[\tau_\delta<\hat{\zeta}]-F(v) \\  &= \sum_e F_e(\delta)p_v(e)\mathbb{P}_v[\tau_\delta<\hat{\zeta}]-F(v)+o(\delta)\end{align*} Now $(\tau_\delta<\hat{\zeta})=(\kappa\ell_\mathcal{C}(\tau_\delta)\leqslant\mathbf{e})=(\kappa\ell_c(\tau_\delta)\leqslant \mathbf{e})$ $\mathbb{P}_v$-a.s., because starting from $v$, only the local time at $v$ can contribute to $\ell_\mathcal{C}(t)$ before $\tau_\delta$. Also, $\ell_v(\tau_\delta)$ has an exponential distribution under $\mathbb{P}_v$, as explained below; we can compute its $\mathbb{E}_v$-mean as $\mathbb{E}_v[\tau_\delta]=\delta$ using the methods in Section \ref{examples}. Therefore $\mathbb{P}_v[\kappa\ell_v(\tau_\delta)\leqslant \mathbf{e}]$ is the probability that a mean $\kappa\delta$ exponential is less than an independent mean $1$ exponential; this is easily found to be $\frac{1}{1+\kappa\delta}$. Therefore the numerator equals \begin{align*}
\frac{1}{1+\kappa\delta}\sum_e F_e(\delta)p_v(e) - F(v)+o(\delta) &= \frac {1}{1+\kappa\delta}\sum_e p_v(e)\left[F_e(\delta-F(v)-\kappa\delta F(v)\right] \\
&= \frac{1}{1+\kappa\delta}\sum_e p_v(e)\left[D_eF(v)-\kappa F(v)\right]\delta+o(\delta)\end{align*} Existence and finiteness of the limit in Dynkin's formula therefore requires that $\sum_e (D_e F)(v)-\kappa F(v)=0$ if $F\in D(\hat{L})$. Similar reasoning away from the vertex shows that $\hat{L}$ has to act as $L$ there. For the rest of the details see \cite{KPS12}. 
\end{proof}

\begin{remark}  The upshot of these considerations is that, in the limit as $h\downarrow 0$ and then $\epsilon\downarrow0$, we obtain an expression $\mathbb{E}_x[\exp\{-\kappa\ell_\mathcal{C}(T_a)\}]$ for the survival probability. Therefore the problem of finding explicit formulas for reaction probabilities reduces to the problem of evaluating $\mathbb{E}_x$-moments of $\ell_\mathcal{C}(T_a)$. There is a method for dealing with this issue in considerable generality, called Kac's moment formula, which we explain in the next section. 
\end{remark}

\section{Explicit formulas for reaction probability} 
We now focus on  the dependence on $\kappa$ of the conversion probability $\alpha_\kappa(x)=1-\psi_k(x)$ for the reaction-diffusion process on a metric graph $\Gamma$, where the active sites consist of a set of vertices $\mathcal{C}\subset \mathcal{V}$.  In this case, it is possible to obtain reasonably explicit formulas for $\alpha_\kappa(x)$ by simple arguments using the strong Markov property. First we consider the case where $\mathcal{C}$ is a single vertex:

\begin{proposition}\label{simpleprop}
Suppose $\mathcal{C}=\{c\}$ and $T_a=\inf\left\{t>0\,:\,X_t\in\mathcal{V}_a\right\}$. Set $\lambda:=\mathbb{E}_{c}\left[\ell_c(T_a)\right]$ = the expected local time accumulated at $c$ up to time $T$, starting from $c$. Then
\begin{equation}\label{simple} {\alpha_\kappa(x)}={\alpha_\infty(x)}\frac{\lambda \kappa}{1+\lambda\kappa}.\end{equation}
\end{proposition}
\begin{proof}
Write $T_c:=\inf \{t:X_t = c\}$, the first time that $X_t$ hits the active vertex. Since $\ell_c(T_a)=0$ on the event  $(T_c \geqslant T_a)$, we have 
\begin{equation}\label{pxc}
\begin{aligned}
\alpha_\kappa(x)&=\mathbb{E}_{x}\left[\left(1-\exp\left\{ -\kappa \ell_c(T_a)\right\} \right)\mathbbm{1}_{(T_c<T_a)}\right]\\
&=\alpha_\infty(x)-\mathbb{E}_{x}\left[\exp\left\{ -\kappa \ell_c(T_a)\right\} \mathbbm{1}_{(T_c<T_a)}\right]
\end{aligned}
\end{equation}
where $\alpha_\infty(x)=\mathbb{E}_{x}[\mathbbm{1}_{(T_c<T_a)}]$ is the probability that $X_{t}$ hits
$c$ before  $\mathcal{V}_a$. Also, since $\ell_c(t)$ does not start increasing until $T_c$, we have that $\ell_c(T_a)=\ell_c(T_a)\circ\theta_{T_c}$ on $(T_c<T_a)$. Using the strong Markov property, and then the fact that $X(T_c)=c$ a.s., we find that the right-most term in the second line of \eqref{pxc} equals: 
\begin{equation*}\begin{aligned}
\mathbb{E}_{x}\left[\exp\left\{-\kappa\ell_c(T_a)\circ\theta_{T_c}\right\}\mathbbm{1}_{(T_c<T_a)}\right] &=
\mathbb{E}_x\left[\mathbb{E}_{X(T_c)}[\exp\{-\kappa\ell_c(T_a)\}]\mathbbm{1}_{(T_c<T_a)}\right]\\
&= \mathbb{E}_c[\exp\{-\kappa\ell_c(T_a)\}]\mathbb{E}_x[\mathbbm{1}_{(T_c<T_a)}]
\end{aligned}\end{equation*} We arrive at 
\begin{equation}\label{prefactor}\alpha_\kappa(x)=\alpha_\infty(x)\left(1-\mathbb{E}_c\left[\exp\left\{ -\kappa \ell_c(T_a)\right\} \right]\right).\end{equation}
Now, we use the fact that {\em $\ell_c(T_a)$ must have an exponential distribution under $\mathbb{P}_c$. } This follows from a simple argument which can be found on p. 106 of \cite{MR06}. Writing $\lambda:=\mathbb{E}_c\left[\ell_c(T_a)\right]$ for the mean, and then explicitly computing the Laplace transform, we obtain the result 
$$
\mathbb{E}_c\left[\exp\left\{ -\kappa \ell_c(T_a) \right\} \right]=\frac{1}{1+\lambda\kappa}.
$$
Inserting this last expression into \eqref{prefactor} establishes \eqref{simple}.
\end{proof}

When $\mathcal{C}$ consists of more than one point, the survival function takes the form \[ \psi_\kappa(x) = \mathbb{E}_x\left[ \exp\left\{  -\kappa\sum_{c\in\mathcal{C}} \ell_c(T_a)\right\}\right] \] where $\ell_c(T_a)$ is the local time at the vertex $c$ up to the exit time $T_a$. In this case, it is still possible to obtain an explicit formula for $\psi_\kappa(x)$ (hence $\alpha_\kappa(x)$) because Kac's moment formula can be used to evaluate the expectation appearing in $\psi_\kappa(x)$. 

Specifically, we can use the following corollary of Kac's moment formula, which can be found in Section 6 of \cite{FP99}:

\begin{proposition} Write $\mathcal{C}=\{c_j\}_{j=1}^C$ with $C=\#\mathcal{C}$ and define the Green's function by \[ G_{ij} := g(c_i,c_j)=\mathbb{E}_{c_i}\left[ \ell_{c_j}(T_a) \right]. \] Let $\kappa_i=\kappa(c_i)$ be a positive function on $\mathcal{C}$. Then the function \[ \psi_\kappa(c_i) := \mathbb{E}_{c_i}\left[\exp\{-\sum_{j=1}^C \kappa_j \ell_{c_j}(T_a)\}\right] \] is the unique solution to the system of equations 
\[ \psi(c_i) = 1 - \sum_{j=1}^C G_{ij}\kappa_j \psi(c_j).\] In other words, if $M_\kappa$ is the diagonal matrix with the entries of $\kappa_j$ along the main diagonal, then \begin{equation}\label{conversion1} \psi_\kappa = (I + GM_\kappa)^{-1}\mathbbm{1}_\mathcal{C}.\end{equation} \end{proposition}

\begin{remark} The expression \eqref{conversion1} for the survival function on $\mathcal{C}$ can be recast as a formula involving determinants by using Cramer's rule. To wit, let $G^{(j)}$ denote the matrix obtained from $G$ by subtracting row $j$ from every row. (So, $G^{(j)}$ has a row of 0's in the $j$-th row.) Then \begin{equation}\label{conversion2} \psi_\kappa(c_j) = \frac{\det(I+G^{(j)}M_\kappa)}{\det(I+GM_\kappa)}.\end{equation} See \cite{MR06} Ch. 2 for additional details about \eqref{conversion2}.  \end{remark} 

Now we can repeat the same argument that was given in Proposition \ref{simpleprop}, replacing the single point $c$ with the set $\mathcal{C}$, so that $T_c$ becomes $T_\mathcal{C}=\inf\{t>0 : X_t\in\mathcal{C}\}$. Then \eqref{pxc} takes the form 
\[ \alpha_\kappa(x)=\alpha_\infty(x)-\mathbb{E}_{x}\left[\exp\left\{ -\kappa \ell_\mathcal{C}(T_a)\right\}\circ \theta_{T_\mathcal{C}} \mathbbm{1}_{(T_\mathcal{C}<T_a)}\right] \] where we abbreviate $\sum_{j=1}^C\ell_{c_j}(T_a)$ as $\ell_\mathcal{C}(T_a)$. To deal with the expectation in the last display, split the event $(T_\mathcal{C}<T_a)$ as $\bigcup_{j=1}^C (X(T_\mathcal{C})=c_j, T_\mathcal{C}<T_a)$, and then write \[ p_j(x)=\mathbb{P}_x[X(T_\mathcal{C})=c_j\,|\, T_\mathcal{C}<T_a].\] Thus, $p_j(x)$ is the probability of starting from $x$ and first hitting $\mathcal{C}$ at $c_j$, conditional on hitting $\mathcal{C}$ at all.  We get 
\begin{equation*} \begin{aligned} 
\mathbb{E}_{x}\left[\exp\left\{ -\kappa \ell_\mathcal{C}(T_a)\right\}\circ \theta_{T_\mathcal{C}} \mathbbm{1}_{(T_\mathcal{C}<T_a)}\right]  &= \sum_{j=1}^C \mathbb{P}_x[X(T_\mathcal{C})=c_j,\,T_\mathcal{C}<T_a]\mathbb{E}_{c_j} \left[ \exp\{-\kappa\ell_\mathcal{C}(T_a)\right] \\
&= \sum_{j=1}^C p_j(x) \mathbb{P}_x[T_{\mathcal{C}}<T_a]\mathbb{E}_{c_j}[\exp\{-\kappa\ell_\mathcal{C}(T_a)\}]
 \end{aligned}\end{equation*} Since $\mathbb{P}_x[T_\mathcal{C}<T_a]=\alpha_\infty(x)$, we can apply \eqref{conversion2} with $\kappa(c_j)$ all being the same constant $\kappa$ to express this last equation in the simple form \begin{equation}\label{conversion3} \alpha_\kappa(x) = \alpha_\infty(x)\left[ 1 - \sum_{j=1}^C p_j(x) \frac{\det\left(I+\kappa G^{(j)}\right)}{\det(I+\kappa G)}\right].\end{equation}
In the numerator, the determinant is a polynomial with degree $\leqslant C-1$, because $\kappa$ appears in only $C-1$ rows of the matrix. On the other hand the denominator is a polynomial with degree $\leqslant C$. Furthermore the coefficients depend only on the $G_{ij}$, which in turn depend only on the geometric properties of $\Gamma$. This explains the claim made in Theorem \ref{theorem1}.

\section{Remark concerning the examples} \label{examples} In this section we present two slightly different approaches to computing the reaction probabilities. In both subsections, we let $X=(X_t)$ be a diffusion process on $\Gamma$ as described in section \ref{fatgraph}. For simplicity, we'll assume that $X$ is a Brownian motion with coefficient $\mathcal{D}$, so that the operator $L_\Gamma$ acts as $L_\Gamma f(x) = \mathcal{D}\frac{d}{dx^2} f_e(x)$ for $x\in E_e^\circ$. This $X$ is a conservative diffusion process (i.e. without killing) which corresponds to the limit of a conservative diffusion process on $U^\epsilon$ as $\epsilon\downarrow 0$. 

\subsection{Method 1}
Previously, we explained how Kac's moment formula can be used to express the reaction probabilities in terms of certain polynomials involving the coefficients $G_{ij} = \mathbb{E}_{c_i}[\ell_{c_j}(T_a)]$. What remains, then, is to compute the $G_{ij}$'s from the graph and diffusion coefficients. To this end we apply the graph stochastic calculus from \cite{FS00}. 

 Then we have, for each $v\in\mathcal{V}$, a process $\ell_v(t)$ which is adapted to the filtration of $X$, continuous, increasing, and increases only on $\{t \, : \, X_t\in v\}$. As explained in \cite{FS00}, $\ell_v(t)$ can be recovered from $X$ by the formula \[ \ell_v(t) = \frac{1}{\mathcal{D}} \lim_{\delta\downarrow 0} \frac{1}{\delta} \int_0^t \mathbbm{1}_{B(v,\delta)}(X_s)\,ds \] where $B(v,\delta)$ is a ball around $v$ of radius $\delta$. Also, let $C^2_b(\Gamma)$ be the set of $f\in C(\Gamma)$ with two continuous derivatives in $\Gamma^\circ$. (The derivatives don't have to extend to be continuous at the vertices.) Then \cite{FS00} gives the following version of the Ito-Tanaka formula for $\Gamma$:

\begin{theorem} Let $F\in C^2_b(\Gamma)$. Define, for each $v\in\mathcal{V}$, \[ \rho_v(F) = \sum_{e\in\mathcal{E} : s(e) = v} p_v(e)(D_e F)(v).\] Then \[ F(X_t) - F(X_0) = M_t + \int_0^t (L_\Gamma F)(X_s)\,ds + \sum_{v\in\mathcal{V}} \rho_v(F) \ell_v(t) \]  where $M_t$ is a continuous local martingale.  \end{theorem}

Ito's formula is all we need to find the $G_{ij}$. Namely, to find $\mathbb{E}_{c_i}[\ell_{c_j}(T_a)]$, we must find a function $F\in C^2_b(\Gamma)$ for which (i) $L_\Gamma F=0$ on $\Gamma^\circ$, (ii) $F(v)=0$ for $v\in\mathcal{V}_a$, and (iii) $\rho_v(F)=0$ except for $v=c_i,c_j$. In this case, (i) means that $F$ is affine on each edge, so that $F_e(y)=a_e y+b_e$ for certain constants $a_e$ and $b_e$ and $0\leqslant y\leqslant |e|$, where we abbreviate $|E_e|=|e|$ for the length of the segment $E_e$. Choosing an orientation arbitrarily for each $e$, this gives a total of $2E$ constants, with $E=\#\mathcal{E}$, which must be determined compatibly with (i)-(iii). This reduces the problem to a discrete one on the underlying combinatorial graph. This method is best illustrated by working a few examples. In the examples, when it is not necessary to give names to the edges, we write $[v_1,v_2]$ for the edge with $s(e)=v_1$ and $t(e)=v_2$.   

\begin{example} Consider the graph shown in the upper right of Figure \ref{diagrams}. Write $c$ for the active vertex in the middle, $a$ for the absorbing vertex at the far end, and $v_1,\ldots,v_n$ for the remaining inert vertices. For simplicity, we assume that the $p_c([c,v_j])$'s are all equal to $\frac{1}{n}$, i.e. that the relative radii were all equal. Starting from any $v_j$, the probability of hitting $c$ before $a$ is 1. So the conversion will have the form \[\alpha_\kappa(v_j) = \frac{\lambda\kappa}{1+\lambda\kappa}\] where $\lambda=\mathbb{E}_c[\ell_c(T_a)]$. Orient the various edges so that $c$ is the origin. Then we seek a function $F$ which is linear on each segment and continuous on $\Gamma$. The condition $\rho_{v_j}(F)=0$ forces $F$ to be constant on $[c,v_j]$ and we can choose this constant to be $l$, the length of $[c,a]$. The condition $F(a)=0$ means that $F$ is linear with slope $-1$ on $[c,a]$. In other words, $\rho_c(F)=\frac{-1}{n}$. Taking expectations in Ito's formula, \[ 0 - l = 0+\frac{-1}{n}\cdot \lambda \quad\Rightarrow \quad \lambda = nl.\] From this we obtain the formula given in item 2 in the list of examples after Theorem \ref{theorem1}: $$\alpha_\kappa(v_j)=\frac{nl\kappa}{1+nl\kappa}.$$ \end{example}

\begin{example} In a similar spirit, consider the second figure in the left column of Figure \ref{diagrams}. Let $v$ be the inactive vertex at the left, and $c_1,\ldots,c_m$ the active vertices, and $a$ the absorbing vertex at the right. Starting from $v$, $X$ hits $\mathcal{C}$ with probability $1$, and with probability $1$ this first hit occurs at $c_1$. Thus $p_1(v)=1$ and $p_j(v)=0$ for the other $c_j$ and formula \eqref{conversion3} simplifies to \[ \alpha_\kappa(v) =1 - \frac{\det(I + \kappa G^{(1)})}{\det(I+\kappa G)}.\] To simplify, we again assume that $p_{[c_j,c_k]}(e)=1/2$ for adjacent $c_j,c_k$ and $p_v([v,c_1])=1$. Now, to compute $G_{ij}$, first write $l_j$ for the length of the segment to the right of $c_j$. Then set $L_j=\sum_{i=j}^m l_i$. Thus, $L_k$ is the sum of the lengths of all the segments to the right of $c_k$.  Regarding $\Gamma$ as a single straight line, take $F$ to be the function which is constantly $L_j$ on $[0,c_j]$ and then decreases linearly with slope $-1$ on $[c_j,a]$, so that $F(a)=0$ and $\rho_{c_j}(F)=-1/2$. Using this $F$ in Ito's formula shows that: if $c_i\leqslant c_j$ then $G_{ij}=2L_j$, and if $c_i > c_j$ then $G_{ij} = 2L_i$. In particular, the matrix $G^{(1)}$ must be strictly triangular, so that the determinant in the numerator of $\alpha_\kappa(v)$ is $1$. Therefore we can write \[ \alpha_\kappa(v) = \frac{\det(I+\kappa G)-1}{\det(I+\kappa G)}\]
\end{example} 

\subsection{Method 2}

A different approach is to work instead with the process $Y=(Y_t)$ obtained from $X$ by sending it to $\triangle$ at $\zeta$, the first time that $\kappa \ell_\mathcal{C}(t)$ exceeds an independent rate 1 exponential. This new process $Y$ is a non-conservative diffusion whose generator still acts as $L_\Gamma$, but whose domain now consists of functions $F\in C^2(\Gamma)$ satisfying these vertex conditions:
\begin{equation}\label{vertices2} \begin{aligned} \sum_{e:s(e)=c} p_v(e)(D_e F)(c)&=\kappa F(c)\quad &c&\in\mathcal{C} \\
\sum_{e:s(e)=v} p_v(e)(D_eF)(v) &= 0 \quad &v&\in\mathcal{V}\setminus[\mathcal{C}\cup\mathcal{V}_a] \end{aligned}\end{equation} It follows that if $F\in C^2_b(\Gamma)$ satisfies $L_\Gamma F(x)=0$ in $\Gamma^\circ$, $F(a)=1$ for $a\in \mathcal{V}_a$, together with the vertex conditions \eqref{vertices2}, then $F(Y_t)$ is a martingale. By optional stopping, \[ F(v) = \mathbb{E}_v [ F( Y(T_a) ) ] = \mathbb{P}_v[ T_a < \zeta ] \] which means that $F(v)=\psi_\kappa(v)$, i.e. the survival function evaluated at $v$. 

Again we assume that $p_v(e)=1/\mathrm{deg}(v)$ whenever $v=s(e)$; equivalently, that all $r_e$'s are equal. Since $F$ must be affine on each edge, it is determined by its values on $\mathcal{V}$, and its derivative $DF$ can be regarded as a function on $\mathcal{E}$: \[ DF(e) = \frac{F(t(e))-F(s(e))}{|e|}.\] Therefore we can couch the problem of determining $F$ as a kind of discrete boundary problem on the combinatorial graph $\mathcal{G}$ rather than on the metric graph $\Gamma$. For this purpose, regard the vertices in $\mathcal{V}_a$ as the {\em boundary} of $\mathcal{G}$, and the remaining vertices as the {\em interior} of $\mathcal{G}$.  Then $F$ is determined by the equations
\begin{equation}
\sum_{e: s(e)=v}{|e|^{-1}}{F(r(e))}=\left({\text{deg}(v)\kappa}_v+\sum_{e: s(e)=v}{|e|^{-1}}\right)F(v)\label{keyequation}
\end{equation}
for interior vertices  and $f(v)=1$ for exit vertices, $v\in \mathcal{V}_a$. Reaction probability for  the examples given in the introduction are easily obtained by solving  the above system of linear equations.

\begin{remark} Equation \eqref{keyequation} can be regarded as a kind of discrete Feynman-Kac equation involving the discrete Laplacian on $\mathcal{G}$. See \cite{bk13} for more information.  \end{remark}
\bibliographystyle{plain}
\bibliography{./biblio}

\begin{thebibliography}{10}

\bibitem{AK12}
Sergio Albeverio and Seiichiro Kusuoka.
\newblock Diffusion processes in thin tubes and their limits on graphs.
\newblock {\em Ann. Probab.}, 40(5):2131--2167, 2012.

\bibitem{B98}
Richard~F. Bass.
\newblock {\em Diffusions and elliptic operators}.
\newblock Probability and its Applications (New York). Springer-Verlag, New
  York, 1998.

\bibitem{BB06}
Richard~F. Bass and Krzysztof Burdzy.
\newblock Pathwise uniqueness for reflecting {B}rownian motion in certain
  planar {L}ipschitz domains.
\newblock {\em Electron. Comm. Probab.}, 11:178--181 (electronic), 2006.

\bibitem{BB08}
Richard~F. Bass and Krzysztof Burdzy.
\newblock On pathwise uniqueness for reflecting {B}rownian motion in
  {$C^{1+\gamma}$} domains.
\newblock {\em Ann. Probab.}, 36(6):2311--2331, 2008.

\bibitem{BH91}
Richard~F. Bass and Pei Hsu.
\newblock Some potential theory for reflecting {B}rownian motion in {H}\"older
  and {L}ipschitz domains.
\newblock {\em Ann. Probab.}, 19(2):486--508, 1991.

\bibitem{bk13}
Gregory Berkolaiko and Peter Kuchment.
\newblock {\em Introduction to quantum graphs}, volume 186 of {\em Mathematical
  Surveys and Monographs}.
\newblock American Mathematical Society, Providence, RI, 2013.

\bibitem{C93}
Zhen~Qing Chen.
\newblock On reflecting diffusion processes and {S}korokhod decompositions.
\newblock {\em Probab. Theory Related Fields}, 94(3):281--315, 1993.

\bibitem{FCYG09}
R.~Feres, A.~Cloninger, G.S. Yablonsky, and J.T. Gleaves.
\newblock A general formula for reactant conversion over a single catalyst
  particle in tap pulse experiments.
\newblock {\em Chemical Engineering Science}, 64(21):4319--4460, 2009.

\bibitem{FYMB09}
R.~Feres, G.S. Yablonsky, A.~Mueller, A.~Baernstein, X.~Zheng, and J.T.
  Gleaves.
\newblock Probabilistic analysis of transport-reaction processes over catalytic
  particles: Theory and experimental testing.
\newblock {\em Chemical Engineering Science}, 64(3):568 -- 581, 2009.

\bibitem{FSYW15}
Renato Feres, Matt Wallace, Gregory Yablonsky, and Ari Stern.
\newblock Explicit formulas for reaction probability in reaction-diffusion
  experiments.
\newblock {\em Submitted}, 2015.

\bibitem{FP99}
P.~J. Fitzsimmons and Jim Pitman.
\newblock Kac's moment formula and the {F}eynman-{K}ac formula for additive
  functionals of a {M}arkov process.
\newblock {\em Stochastic Process. Appl.}, 79(1):117--134, 1999.

\bibitem{F84}
Robert~L. Foote.
\newblock Regularity of the distance function.
\newblock {\em Proc. Amer. Math. Soc.}, 92(1):153--155, 1984.

\bibitem{FS00}
Mark Freidlin and Shuenn-Jyi Sheu.
\newblock Diffusion processes on graphs: stochastic differential equations,
  large deviation principle.
\newblock {\em Probab. Theory Related Fields}, 116(2):181--220, 2000.

\bibitem{FW93}
Mark~I. Freidlin and Alexander~D. Wentzell.
\newblock Diffusion processes on graphs and the averaging principle.
\newblock {\em Ann. Probab.}, 21(4):2215--2245, 1993.

\bibitem{GYZF10}
John~T. Gleaves, Gregory Yablonsky, Xiaolin Zheng, Rebecca Fushimi, and
  Patrick~L. Mills.
\newblock Temporal analysis of products ({TAP}) -- recent advances in
  technology for kinetic analysis of multi-component catalysts.
\newblock {\em Journal of Molecular Catalysis A: Chemical}, 315(2):108 -- 134,
  2010.
\newblock In memory of M.I. Temkin.

\bibitem{IW}
Nobuyuki Ikeda and Shinzo Watanabe.
\newblock {\em Stochastic differential equations and diffusion processes},
  volume~24 of {\em North-Holland Mathematical Library}.
\newblock North-Holland Publishing Co., Amsterdam; Kodansha, Ltd., Tokyo,
  second edition, 1989.

\bibitem{KPS12}
Vadim Kostrykin, J{\"u}rgen Potthoff, and Robert Schrader.
\newblock Brownian motions on metric graphs.
\newblock {\em J. Math. Phys.}, 53(9):095206, 36, 2012.

\bibitem{LS84}
P.-L. Lions and A.-S. Sznitman.
\newblock Stochastic differential equations with reflecting boundary
  conditions.
\newblock {\em Comm. Pure Appl. Math.}, 37(4):511--537, 1984.

\bibitem{MR06}
Michael~B. Marcus and Jay Rosen.
\newblock {\em Markov processes, {G}aussian processes, and local times}, volume
  100 of {\em Cambridge Studies in Advanced Mathematics}.
\newblock Cambridge University Press, Cambridge, 2006.

\bibitem{RW1}
L.~C.~G. Rogers and David Williams.
\newblock {\em Diffusions, {M}arkov processes, and martingales. {V}ol. 1:
  Foundations}.
\newblock Cambridge Mathematical Library. Cambridge University Press,
  Cambridge, 2000.
\newblock Reprint of the second (1994) edition.

\bibitem{S87}
Yasumasa Saisho.
\newblock Stochastic differential equations for multidimensional domain with
  reflecting boundary.
\newblock {\em Probab. Theory Related Fields}, 74(3):455--477, 1987.

\end{thebibliography}

\end{document}